\documentclass[12pt,a4paper]{article}
\usepackage{latexsym,amssymb,amsfonts,amsmath,amsthm,nccmath,float,enumitem,setspace,bm,graphicx}
\usepackage[dvipsnames]{xcolor}
\usepackage[hidelinks]{hyperref}
\usepackage[margin=2cm]{geometry}

\hypersetup{
    colorlinks,
    linkcolor={red!80!black},
    citecolor={blue!80!black},
    urlcolor={blue!80!black}
}

\renewcommand{\mod}[1]{\pmod{#1}}
\newtheorem{theorem}{Theorem}[section]
\newtheorem{corollary}[theorem]{Corollary}

\newtheorem{lemma}[theorem]{Lemma}

\theoremstyle{definition}

\def \P {\mathcal{P}}
\def \G {\mathcal{G}}
\def \C {\mathcal{C}}

\def \F {\mathcal{F}}

\def \leq {\leqslant}
\def \geq {\geqslant}
\def \ua {^{\,\uparrow}}
\def \da {^{\,\downarrow}}
\def \ud {^{\,\updownarrow}}

\let\oldproofname=\proofname
\renewcommand{\proofname}{\rm\bf{\oldproofname}}

\title{Minimising the total number of subsets and supersets}
\author{Adam Gowty \hspace{4mm} Daniel Horsley \hspace{4mm} Adam Mammoliti \\ School of Mathematics, Monash University, Victoria 3800, Australia}
\date{}

\begin{document}
\setstretch{1.2}
\maketitle

\begin{abstract}
Let $\F$ be a family of subsets of a ground set $\{1,\ldots,n\}$ with $|\F|=m$, and let $\F\ud$ denote the family of all subsets of $\{1,\ldots,n\}$ that are subsets or supersets of sets in $\F$. Here we determine the minimum value that $|\F\ud|$ can attain as a function of $n$ and $m$. This can be thought of as a `two-sided' Kruskal-Katona style result. It also gives a solution to the isoperimetric problem on the graph whose vertices are the subsets of $\{1,\ldots,n\}$ and in which two vertices are adjacent if one is a subset of the other. This graph is a supergraph of the $n$-dimensional hypercube and we note some similarities between our results and Harper's theorem, which solves the isoperimetric problem for hypercubes. In particular, analogously to Harper's theorem, we show there is a total ordering of the subsets of $\{1,\ldots,n\}$ such that, for each initial segment $\F$ of this ordering,  $\F\ud$ has the minimum possible size. Our results also answer a question that arises naturally out of work of Gerbner et al.\ on cross-Sperner families and allow us to strengthen one of their main results.
\end{abstract}

\section{Introduction}

For a set $X$, let $2^X$ denote the power set of $X$. For a nonnegative integer $n$, let $[n]=\{1,\ldots,n\}$ and note in particular that $[n]=\emptyset$ when $n=0$. Let $\F$ be a family of subsets of a ground set $X$. We define $\F\da=\{S \subseteq X: S \subseteq F \text{ for some } F \in \F\}$, $\F\ua=\{S \subseteq X: F \subseteq S \text{ for some } F \in \F\}$ and $\F\ud= \F\ua \cup \F\da$. Note that the latter two definitions depend implicitly on the ground set of $\F$ and we will clarify what ground set we are considering whenever it may be unclear. We say $\F$ is \emph{convex} if $\F\ua \cap \F\da=\F$.

For nonnegative integers $m$ and $n$ such that $m \leq 2^n$, let
\[\Phi(n,m)=\min\left\{\left|\F\ud\right|: \F \subseteq 2^{[n]} \text{ and } |\F|=m\right\}.\]
Note that, for any nonnegative integer $n$ we clearly have $\Phi(n,0)=0$, $\Phi(n,2^n)=2^n$ and that $\Phi(n,m)$ is nondecreasing in $m$. This paper is devoted to the study of the values of $\Phi(n,m)$ and of families $\F$ of subsets of $[n]$ for which $|\F\ud|=\Phi(n,m)$. We first state our main results and then discuss their motivation and context. Our first main result determines $\Phi(n,m)$, for all nonnegative integers $m$ and $n$ such that $m \leq 2^n$, by means of a recursion.

\begin{theorem}\label{T:isoperRecursive}
We have $(\Phi(0,0),\Phi(0,1))=(0,1)$ and $(\Phi(1,0),\Phi(1,1),\Phi(1,2))=(0,2,2)$. For integers $n \geq 2$ and $m \in \{0,\ldots,2^n\}$,
\[
\Phi(n,m) =
\left\{
  \begin{array}{ll}
    2\Phi(n-2,m)+m\quad & \hbox{if $m \in \{0,\ldots,2^{n-2}\}$} \\
    2^n-s & \hbox{if $m \in \{2^{n-2}+1,\ldots,2^n\}$} \\
  \end{array}
\right.\]
where $s$ is the greatest element of $\{0,\ldots,2^{n-2}\}$ such that $\Phi(n,s) \leq 2^n-m$.
\end{theorem}

The determination of $\Phi(n,m)$ given by Theorem~\ref{T:isoperRecursive} has drawbacks: it does not immediately suggest any neat approximation for $\Phi(n,m)$ and, naively, it seems that computing $\Phi(n,m)$ for a given $n$ and $m$ may require computing $\Theta(2^n)$ other values of $\Phi$. We are able to remedy both these problems by giving a closed form approximation for $\Phi(n,m)$ and an efficient means for computing $\Phi(n,m)$.

\begin{corollary}\label{C:isoperBounds}
Let $n$ and $m$ be nonnegative integers with $m \leq 2^n$. Then
\[\Phi(n,m) \geq f(n,m) \qquad \text{ where } \qquad f(n,m)=\sqrt{2^{n+2}m}-m\]
and we have equality whenever $f(n,m)$ is an integer. Furthermore $\Phi(n,m) \leq f(n,m)+\sqrt{2^n}$ and hence $\Phi(n,m) \sim f(n,m)$ as $n \rightarrow \infty$ with $m=\omega(1)$.
\end{corollary}

\begin{theorem}\label{T:isoperQuick}
Let $n \geq 0$ be an integer and let $m \in \{1,\ldots,2^n\}$, let $\kappa=1$ if $n$ is even and $\kappa=2$ if $n$ is odd, and let $c$ be the positive integer such that $\kappa c(c-1) \leq m < \kappa c(c+1)$. Then
\[\Phi(n,m)=\sqrt{\kappa 2^{n}}\Bigl(2c-1+2\delta_{2 \kappa c}\bigl(m-\kappa c(c-1)\bigr)\Bigr)-m\]
where, for any positive integer $k$, $\delta_k: \{0,\ldots,k\} \rightarrow \{y \in \mathbb{Q}:0\leq y \leq 1\}$ is a function recursively defined by $\delta_1(x)=x$ and, for $k \geq 2$,
\begin{equation}\label{E:deltaDef}
\delta_k(x) =
\begin{cases}
\frac{1}{2}\delta_{\lfloor k/2\rfloor}(x) & \text{if $x \in \{0,\ldots,\lfloor\frac{k}{2}\rfloor-1\}$} \\
\frac{1}{2}+\frac{1}{2}\delta_{\lceil k/2 \rceil}(x-\lfloor\frac{k}{2}\rfloor)  &\text{if $x \in \{\lfloor\frac{k}{2}\rfloor,\ldots,k\}$.}
\end{cases}
\end{equation}
\end{theorem}

Observe that evaluating $\delta_k(x)$ involves at most $\log_2(k)$ recursive steps and hence Theorem~\ref{T:isoperQuick} allows $\Phi(n,m)$ to be computed in $O(n)$ time. Values of $\Phi(n,m)$ for small $n$ are given in Table~\ref{Tab:smallVals} and an illustration of these values is given in Figure~\ref{F:Ferrers}. Our last main result says that there is a total ordering of the subsets of $[n]$ such that, for each $m \in \{0,\ldots,2^n\}$, $|(\F_m)\ud|=\Phi(n,m)$, where $\F_m$ is the initial segment containing the first $m$ sets under this ordering. Furthermore we show that certain special families arise as initial segments of this ordering.

\begin{theorem}\label{T:totalOrdering}
Let $n$ be a nonnegative integer. There is a chain $\F_0 \subsetneq \cdots \subsetneq \F_{2^n}$ of convex families of subsets of $[n]$ such that $|\F_m|=m$ and $|(\F_m)\ud|=\Phi(n,m)$ for each $m \in \{0,\ldots,2^{n}\}$. Furthermore, this chain includes the families $\C_{n,a}$ and $\C^*_{n,a}$ for each $a \in \{0,\ldots,n-2\}$ with $a \equiv n \mod{2}$, where
\begin{align*}
\mathcal{C}_{n,a} & = \{A \subseteq [n]:[\tfrac{n-a}{2}] \subseteq A \subseteq [\tfrac{n+a}{2}]\} \\
\mathcal{C}^*_{n,a} &= \{A \subseteq [n]: A \nsubseteq \{\tfrac{n-a}{2}+1,\ldots,n\} \textup{ and } \{\tfrac{n+a}{2}+1,\ldots,n\} \nsubseteq A\}.
\end{align*}
\end{theorem}

The investigation undertaken here is motivated in at least three different ways. The first of these relates to the Kruskal-Katona theorem \cite{Kat,Kru} (see also \cite[\S5]{Bol}). The \emph{lower shadow} of a family $\F$ of $k$-subsets of $[n]$ is defined to be the family of all $(k-1)$-subsets of $[n]$ that are a subset of at least one set in $\F$ and the \emph{upper shadow} of $\F$ is the family of all $(k+1)$-subsets of $[n]$ that are a superset of at least one set in $\F$. The Kruskal-Katona theorem states that, over all families $\F$ of $k$-subsets of $[n]$ with $|\F|=m$, the size of the lower shadow of $\F$ is minimised when $\F$ is taken to be the first $m$ $k$-subsets of $[n]$ in colexicographic order. It is known that this same choice of $\F$ also minimises the size of $\F\da$. Thus our results can be seen as proving a certain `two-sided' variant of the Kruskal-Katona theorem in which $\F$ is allowed to contain sets of different sizes and we are interested in minimising the size of $\F\ud$. Through this lens, Theorem~\ref{T:isoperQuick} can be seen as analogous to the original Kruskal-Katona theorem, which gives an exact answer that can be difficult to work with, and Corollary~\ref{C:isoperBounds} as akin to the neat approximation due to Lov\'{a}sz \cite[p. 95]{Lov}. Bashov \cite{Bas1,Bas2} studied a different two-sided version of the Kruskal-Katona theorem in which $\F$ remained restricted to contain sets of uniform size and the size of the union of the upper and lower shadow (rather than the union of $\F\ua$ and $\F\da$) was to be minimised. Interestingly, he found that no minimising total order of subsets of $[n]$ existed for his problem, contrasting with both the Kruskal-Katona theorem and Theorem~\ref{T:totalOrdering}.

A second motivation concerns vertex-isoperimetric problems for graphs. The vertex boundary of a set of vertices $U$ of a graph $G$ is the set of all vertices of $G$ not in $U$ that are adjacent to at least one vertex in $U$. The vertex-isoperimetric problem on $G$ asks us to determine the minimum size of the vertex boundary of $U$ over all sets $U$ of a given number of vertices. In particular, much attention has been paid to the vertex-isoperimetric problem on the $n$-dimensional hypercube, which can be defined as the graph with vertex set $2^{[n]}$ and edge set $\{AB: A \subsetneq B, |B|=|A|+1\}$. The classical result solving the problem is due to Harper \cite{Har} (see also \cite[\S16]{Bol}) but related problems have been a subject of recent study \cite{KeeLon,PrzRob,Rat}. Harper demonstrated that any initial segment of the so-called simplicial ordering of vertices of the hypercube has minimum boundary size over all subsets of the vertex set with the same size. Furthermore, the $d$-Hamming ball $\{A \subseteq [n]: |A| \leq d\}$ arises as an initial segment of the simplicial ordering for each $d \in \{1,\ldots,n-1\}$. Our results can be seen as a solution to the vertex isoperimetric problem on $\Gamma_n$, the supergraph of the $n$-dimensional hypercube with vertex set $2^{[n]}$ and edge set $\{AB: A \subsetneq B\}$. In this graph, the boundary of a set $\F$ of vertices is $\F\ud \setminus \F$ and hence the minimum boundary size over all sets of $m$ vertices is $\Phi(n,m)-m$. The ordering of the vertices of $\Gamma_n$ given by Theorem~\ref{T:totalOrdering} and the families $\C_{n,a}$ and $\C^*_{n,a}$ can be seen as analogous to the simplicial ordering of vertices of the hypercube and the Hamming balls.

Finally, our results relate to work by Gerbner et al.\ in \cite{GerLemPalPatSze} on pairs of cross-Sperner families. A pair of families of sets is cross-Sperner if no set in one family is a subset of a set in the other. Our work answers a question that arises naturally out of theirs and also allows us to strengthen one of the main results \cite[Theorem 1.1]{GerLemPalPatSze} from their paper (see Theorem~\ref{T:ImproveGerbnerEtAl}). We return to this in Section~\ref{S:CS}.

Very shortly after our paper appeared as a preprint, Behague, Kuperus, Morrison and Wright \cite{BehKupMorWri} posted a preprint containing independent work on closely related questions. They give a very neat proof of the lower bound for $\Phi(n,m)$ given in Corollary~\ref{C:isoperBounds}, and then use this to improve on results from \cite{GerLemPalPatSze} by establishing new upper and lower bounds on the sizes of collections of families that are pairwise cross-Sperner. As we do, they also strengthen \cite[Theorem 1.1]{GerLemPalPatSze} by proving a version of our Theorem~\ref{T:ImproveGerbnerEtAl}. They do not obtain results analogous to our Theorems~\ref{T:isoperRecursive}, \ref{T:isoperQuick} and \ref{T:totalOrdering}, however.

The rest of the paper is organised as follows. In Section~\ref{S:prelim} we establish some concepts, notation and preliminary results that will form the foundation for the proofs of our main results. Section~\ref{S:shift} is devoted to introducing a generalisation of the classical concept of shifting and showing that $|\G\ud| \leq |\F\ud|$ when $\G$ is a family obtained from another family $\F$ by this generalised shifting. In Section~\ref{S:mainProof} we combine results from Sections~\ref{S:prelim} and \ref{S:shift} to complete our proof of Theorem~\ref{T:isoperRecursive} along with Corollary~\ref{C:isoperBounds}. We show in Section~\ref{S:quickProof} that, with some extra work, Theorem~\ref{T:isoperQuick} follows from Theorem~\ref{T:isoperRecursive}. In Section~\ref{S:orderProof} we prove Theorem~\ref{T:totalOrdering} and in Section~\ref{S:CS} we discuss the consequences of our results for pairs of cross-Sperner families. Some brief conclusive comments are given in Section~\ref{S:conc}.

\section{Preliminaries}\label{S:prelim}

In this section we lay the groundwork for our proof of Theorem~\ref{T:isoperRecursive}. For nonnegative integers $n$ and $m$ we call a family of subsets $\F$ of $[n]$ with $|\F|=m$ an \emph{$(n,m)$-family}. If an $(n,m)$-family $\F$ has the property that $|\F\ud|=\Phi(n,m)$, then we call it an \emph{$(n,m)$-witness family}.

Our first step is to show that, for a fixed value of $n$, the values of $\Phi(n,m)$ are related to one another according to a certain self-conjugacy property. More precisely, if the sequence $\Phi(n,1),\Phi(n,2),\ldots,\Phi(n,2^n)$ is viewed as an integer partition, then this partition is self conjugate. See Figure~\ref{F:Ferrers} for an illustration of this and see \cite[Chapter 3]{AndEri} for the relevant definitions related to integer partitions, which will not be required here. This self-conjugacy is the basic idea underlying the second case of the recursion in Theorem~\ref{T:isoperRecursive}.

\begin{lemma}\label{L:selfConjugate}
Let $m$ and $n$ be nonnegative integers such that $m \leq 2^n$. Then $\Phi(n,m) = 2^n-s$ where $s$ is the greatest element of $\{0,\ldots,2^{n}\}$ such that $\Phi(n,s) \leq 2^n-m$.
\end{lemma}

\begin{proof}
Because $\Phi(n,s) \leq 2^n-m$, there is an $(n,s)$-family $\F$ such that $|\F\ud| \leq 2^n-m$. So we can choose a subfamily $\G$ of $2^{[n]} \setminus \F\ud$ with $|\G|=m$. Since no set in $\G$ is in $\F\ud$, no set in $\G\ud$ is in $\F$. So $\G\ud \subseteq 2^{[n]} \setminus \F$ and hence $|\G\ud| \leq 2^n-s$. Thus $\Phi(n,m) \leq 2^n-s$.

Now suppose for a contradiction that $\Phi(n,m) \leq 2^n-s-1$. Then there is an $(n,m)$-family $\mathcal{H}$ such that $|\mathcal{H}\ud| \leq 2^n-s-1$. So we can choose a subfamily $\mathcal{I}$ of $2^{[n]} \setminus \mathcal{H}\ud$ with $|\mathcal{I}|=s+1$. Then $\mathcal{I}\ud \subseteq 2^{[n]} \setminus \mathcal{H}$ and hence $|\mathcal{I}\ud| \leq 2^n-m$. This implies $\Phi(n,s+1) \leq 2^n-m$ in contradiction to the definition of $s$. Thus $\Phi(n,m) \geq 2^n-s$.
\end{proof}

Lemma~\ref{L:selfConjugate} goes most of the way toward establishing the second case of the recurrence in Theorem~\ref{T:isoperRecursive}, but it is not immediately apparent that the value of $s$ defined in Lemma~\ref{L:selfConjugate} will be in the set $\{0,\ldots,2^{n-2}\}$ for all $m \in \{2^{n-2}+1,\ldots,2^n\}$. This will definitely be the case, however, if $\Phi(n,2^{n-2})=3\cdot2^{n-2}$. For any integer $n \geq 2$ we have $\Phi(n-2,2^{n-2})=2^{n-2}$ and hence we will have $\Phi(n,2^{n-2})=3\cdot2^{n-2}$ if the first case of recurrence in Theorem~\ref{T:isoperRecursive} holds. So, given Lemma~\ref{L:selfConjugate}, it suffices to prove the first case to establish the totality of Theorem~\ref{T:isoperRecursive}.

For a family $\F$ of sets, we say a set $F \in \F$ is \emph{minimal in $\F$} if no set in $\F$ is a proper subset of $F$ and is \emph{maximal in $\F$} if no set in $\F$ is a proper superset of $F$. Recall that we say $\F$ is convex if $\F\ua \cap \F\da=\F$.  Equivalently, $\F$ is convex if, for any $F_1,F_2 \in \F$ and any set $A$ such that $F_1 \subseteq A \subseteq F_2$, we have $A \in \F$. Observe that removing a minimal or maximal set from a convex family of sets results in a new family that is still convex. Our next lemma shows that if we are attempting to find values of $\Phi(n,m)$, then it suffices to consider convex families.

\begin{lemma}\label{L:CSNoGaps}
Let $n$ and $m$ be nonnegative integers with $m \leq 2^n$. There exists an $(n,m)$-witness family that is convex.
\end{lemma}

\begin{proof}
Let $\F$ be an $(n,m)$-witness family. We are done if $\F$ is convex, so suppose otherwise that there are sets $F_1,F_2 \in \F$ and $A \in 2^{[n]}\setminus \F$ such that $F_1 \subseteq A \subseteq F_2$. We may further suppose that $F_2$ has been chosen so that it is maximal in $\F$. Let $\mathcal{G}=(\F \setminus \{F_2\}) \cup \{A\}$. Then $\mathcal{G}\ua = \F\ua$ because $F_1$ is in both $\F$ and $\mathcal{G}$ and $\{F_2\}\ua \subseteq \{A\}\ua \subseteq \{F_1\}\ua$ since $F_1 \subseteq A \subseteq F_2$. Further, $\mathcal{G}\da \subseteq \F\da \setminus \{F_2\}$ because no proper superset of $F_2$ is in $\F$ and because $\{A\}\da \subseteq \{F_2\}\da \setminus \{F_2\}$ since $A \subsetneq F_2$. It follows that $|\mathcal{G}\ud| \leq |\F\ud|$ and hence that $|\mathcal{G}\ud|=|\F\ud|=\Phi(n,m)$ by the definition of $\Phi(n,m)$. Furthermore $F_2 \in \G\ua \setminus \G\da$ and hence $|\mathcal{G}\ua \cap \mathcal{G}\da| < |\F\ua \cap \F\da|$. So we can iterate this process until we eventually obtain an $(n,m)$-witness family that is convex.
\end{proof}

Equipped with Lemma~\ref{L:CSNoGaps}, we can also exhibit the idea underlying the first case of the recursion in Theorem~\ref{T:isoperRecursive}. When reading the lemma below it is important to note that $\P\ud$ is defined with respect to the ground set $\{2,\ldots,n-1\}$ while $\F\ud$ is defined with respect to the ground set $[n]$.

\begin{lemma}\label{L:sandwich}
Let $n$ and $m$ be nonnegative integers with $n \geq 2$ and $m \leq 2^{n-2}$. If $\P$ is a convex family of subsets of $\{2,\ldots,n-1\}$ such that $|\P|=m$ and $\F$ is the $(n,m)$-family given by $\F=\{\{1\} \cup P:P \in \P\}$, then $\F$ is convex and $|\F\ud|=2|\P\ud|+m$.
Hence $\Phi(n,m) \leq 2\Phi(n-2,m)+m$.
\end{lemma}

\begin{proof}
Remembering the respective ground sets of $\F$ and $\P$, we have
\begin{equation}\label{E:doubleUpDown}
\F\da=\P\da \cup \bigl\{\{1\} \cup P: P \in \P\da\bigr\} \text{~~and~~} \F\ua=\bigl\{\{1\} \cup P: P \in \P\ua\bigr\} \cup \bigl\{\{1,n\} \cup P: P \in \P\ua\bigr\}.
\end{equation}
From $\P\ua \cap \P\da = \P$ and \eqref{E:doubleUpDown} it follows that $\F\ua \cap \F\da = \F$, so $\F$ is indeed convex. Using $|\F\ua \cap \F\da|=m$ and $|\P\ua \cap \P\da| = m$ together with \eqref{E:doubleUpDown}, we also have
\[|\F\ud|=|\F\ua|+|\F\da|-m=2|\P\ua|+2|\P\da|-m=2|\P\ud|+m.\]
Because Lemma~\ref{L:CSNoGaps} guarantees that there exists a convex family $\P$ of subsets of $\{2,\ldots,n-1\}$ such that $|\P|=m$ and $|\P\ud|=\Phi(n-2,m)$, it follows that $\Phi(n,m) \leq 2\Phi(n-2,m)+m$.
\end{proof}

In view of Lemmas~\ref{L:selfConjugate} and \ref{L:sandwich}, to prove Theorem~\ref{T:isoperRecursive} it only remains to show that $\Phi(n,m) \geq 2\Phi(n-2,m)+m$ when $m \in \{0,\ldots,2^{n-2}\}$. This takes some more work, however. In particular, we require a result saying that we can always find a convex witness family that is invariant under certain kinds of shifting operations. We pursue this in the next section.

\section{Shifting}\label{S:shift}

We make use of a variant of the classical notion of shifting (see \cite[\S5]{Bol}, for example). Most significantly, our notion differs from the usual one in that it allows sets to be replaced with sets of a different size. For a family $\F$ of subsets of $[n]$ and two nonempty disjoint subsets $I$ and $J$ of $[n]$ we define, for each $F \in \F$,
\[S^{\F}_{I,J}(F)=
\left\{
  \begin{array}{ll}
    (F \setminus J) \cup I & \hbox{if $I \cap F = \emptyset$, $J \subseteq F$ and $(F \setminus J) \cup I \notin \F$} \\
    F & \hbox{otherwise.}
  \end{array}
\right.\]
We further define $S_{I,J}(\F)=\{S^{\F}_{I,J}(F):F \in \F\}$. Observe that $|S_{I,J}(\F)|=|\F|$. If $I=\{i\}$ and $J=\{j\}$, this definition agrees with the conventional definition of a shift. We say a family $\F$ of subsets of $[n]$ is \emph{strongly shifted} if $S_{I,J}(\F)=\F$ for all $I,J \subseteq [n]$ such that $\max (I) < \min (J)$. Our goal in this section is to show that, for any admissible $n$ and $m$, there always exists a convex $(n,m)$-witness family that is strongly shifted. We first require some more basic concepts and notation.

Let $\F$ be a family of subsets of $[n]$. For any set $F \in \F$, let $\overline{F}$ denote the set $[n] \setminus F$ and let $\overline{\F}$ denote the family $\{\overline{F} :F \in \F\}$. Note that $\overline{\F}$ does not mean $2^{[n]} \setminus \F$.  Let $\rho$ be the permutation of $[n]$ such that $\rho(x)=n+1-x$ for each $x \in [n]$ and for a subset $F$ of $[n]$, let $\rho(F)=\{\rho(x):x \in F\}$. The \emph{reverse} of $\F$ is the family $\rho(\F)=\{\rho(F):F \in \F\}$. Obviously $|\rho(\F)|=|\F|$. We note some basic properties of complements, reverses and shifts.

\begin{lemma}\label{L:relbetops}
Let $\F$ be a family of subsets of $[n]$ and $I$ and $J$ be nonempty disjoint subsets of $[n]$. Then
\begin{itemize}
\item[\textup{(i)}]
$\left(\,\overline{\F}\,\right)\da = \overline{\F\ua}$, \quad $\left(\,\overline{\F}\,\right)\ua= \overline{\F\da}$,
\quad and \quad $\left(\overline{\F}\right)\ud = \overline{\F\ud}$;
\item[\textup{(ii)}]
$\overline{\rho(\F)}= \rho\left(\overline{\F}\right)$, \quad
$\rho(\F\da)= (\rho(\F))\da$, \quad
$\rho(\F\ua)= (\rho(\F))\ua$, \quad and \quad
$\rho(\F\ud)= (\rho(\F))\ud$;
\item[\textup{(iii)}]
$\overline{S_{I,J}(\F)} = S_{J,I}(\overline{\F})$.
\end{itemize}
\end{lemma}
\begin{proof}
It is a simple exercise to prove (i) by using that $A \subseteq B$ if and only if $\overline{B}\subseteq \overline{A}$ for $A,B \subseteq [n]$. The properties in (ii) are immediate from the fact that $\rho$ is a permutation of the ground set $[n]$.

For (iii), we abbreviate $S_{I,J}^{\F}$ to $S$ and $S_{J,I}^{\overline{\F}}$ to $S'$.
Let $F \in \F$.
Note that $J \cap \overline{F} = \emptyset$ if and only if $J \subseteq F$, that $I \subseteq \overline{F}$ if and only if $I \cap F= \emptyset$, and that $(\overline{F}\setminus I) \cup J \in \overline{\F}$ if and only if $(F\setminus J) \cup I \in \F$,
since $\overline{ (F\setminus J) \cup I} =  (\overline{F}\setminus I) \cup J$.
Thus $S({F})={F}$ if and only if $S'(\overline{F})=\overline{F}$. When $S(F)\neq F$ and $S'(\overline{F})\neq \overline{F}$ we have
\[
\overline{S(F)} =\overline{(F\setminus J) \cup I} =  \left(\overline{F}\setminus I\right) \cup J = S'(\overline{F})\,.
\]
Thus $\overline{S(F)} = S'(\overline{F})$ for all $F \in \F$ and we have $\overline{S_{I,J}(\F)}=S_{J,I}(\overline{\F})$.
\end{proof}

We now prove  that, under certain conditions, a witness family remains a witness family even when a shift is applied to it. This result is analogous to results for conventional shifts that are often used to prove the Kruskal-Katona theorem (see \cite{Fra}, for example).

\begin{lemma}\label{L:shiftfam}
Let $\F$ be a family of subsets of $[n]$ and let $I$ and $J$ be nonempty disjoint subsets of $[n]$. If $S_{I',J}(\F)=\F$ for all nonempty proper subsets $I'$ of $I$ and $S_{I,J'}(\F)=\F$ for all nonempty proper subsets $J'$ of $J$, then
\begin{itemize}
    \item[\textup{(i)}]
$(S_{I,J}(\F))\da \subseteq S_{I,J}(\F\da)$ and hence $|(S_{I,J}(\F))\da| \leq |\F\da|$;
    \item[\textup{(ii)}]
$(S_{I,J}(\F))\ua \subseteq S_{I,J}(\F\ua)$ and hence $|(S_{I,J}(\F))\ua| \leq |\F\ua|$.
\end{itemize}
\end{lemma}

\begin{proof}

\noindent {\textbf{(i).}} We abbreviate $S_{I,J}^{\F}$ to $S$ and $S^{\F\da}_{I,J}$ to $S'$. Suppose for a contradiction that there is a set $A$ in $(S_{I,J}(\F))\da \setminus S_{I,J}(\F\da)$. Since $A \in (S_{I,J}(\F))\da$, we have that $A \subseteq S(F)$ for some $F \in \F$. We consider two cases according to whether $S(F)=F$.

Suppose first that $S(F) \neq F$.
Then $J \subseteq F$, $I \cap F = \emptyset$ and $S(F)=(F \setminus J) \cup I \notin \F$.
Further, $J \cap A = \emptyset$ since $A \subseteq S(F)$. We must have $I \nsubseteq A$ for otherwise the set $B=(A \setminus I) \cup J$ would be a subset of $F$ since $A \subseteq S(F)$, and the fact that $B \in \F\da$ would contradict our assumption that $A \notin S_{I,J}(\F\da)$. We will show that $A \in \F\da$ and then, since $J \cap A = \emptyset$, we will have $S'(A)=A$ contradicting $A \notin S_{I,J}(\F\da)$. If $A \cap I=\emptyset$, then $A \subseteq F$ and so $A \in \F\da$. If $A \cap I \neq \emptyset$, then by our hypotheses and since $ I \nsubseteq A$, we have $S_{I',J}(\F)=\F$ where $I'=I \cap A$. Thus $(F\setminus J) \cup I' \in \F$, which implies that $A \in \F\da$.

Now suppose that $S(F)=F$. Then $A \subseteq F$ and therefore $A \in \F\da$. So, since $A \notin S_{I,J}(\F\da)$, we must have $S'(A) \neq A$. So $J \subseteq A$, $I \cap A = \emptyset$ and $S'(A)=(A \setminus J) \cup I \notin \F\da$. The first of these facts implies that $J \subseteq F$ since $A \subseteq F$ and the last of these facts implies that $I \nsubseteq F$ for otherwise $S'(A)=(A \setminus J) \cup I \subseteq F$ contradicting $ S'(A)\notin \F\da$. If $I \cap F = \emptyset$ then, since $J \subseteq F$ and $S(F)=F$, the set $(F \setminus J) \cup I$ must be in $\F$, contradicting $S'(A)=(A \setminus J) \cup I \notin \F\da$. If $I \cap F \neq \emptyset$, then by our hypotheses and since $I \nsubseteq F$, we have $S_{I',J}(\F) =\F$ where $I' = I \setminus F$. Hence, we have that $(F\setminus J) \cup I = (F \setminus J) \cup I' \in \F$,
and so $S'(A)=(A \setminus J) \cup I \subseteq (F \setminus J) \cup I$, contradicting $S'(A)\notin \F\da$. \smallskip

\noindent {\textbf{(ii).}} By using (i) and (iii) of Lemma~\ref{L:relbetops} and applying (i) of this lemma, we have
\[
\overline{(S_{I,J}(\F))\ua} =  \bigl(\,\overline{S_{I,J}(\F)}\,\bigr)\da = \bigl(S_{J,I}\bigl(\overline{\F}\bigr)\bigr)\da \subseteq
  S_{J,I}\bigl(\bigl(\,\overline{\F}\,\bigr)\da\bigr) = S_{J,I}\bigl(\,\overline{\F\ua}\,\bigr) = \overline{S_{I,J}(\F\ua)}
\]
and hence $(S_{I,J}(\F))\ua \subseteq S_{I,J}(\F\ua)$. Thus (ii) follows from (i).
\end{proof}

Using Lemmas~\ref{L:CSNoGaps} and \ref{L:shiftfam}, we can now achieve the goal of this section.

\begin{lemma}\label{L:canMinFamSF}
Let $n$ and $m$ be nonnegative integers with $m \leq 2^n$. There exists an $(n,m)$-witness family $\F$ that is convex and strongly shifted.
\end{lemma}

\begin{proof}
By Lemma~\ref{L:CSNoGaps} there is a convex $(n,m)$-witness family $\F$. We are done if $\F$ is strongly shifted, so suppose otherwise. Then there are subsets $I$ and $J$ of $[n]$ such that $\max(I) < \min(J)$ and $S_{I,J}(\F) \neq \F$. We may further suppose that $I$ and $J$ have been chosen to be minimal in the sense that $S_{I',J}(\F)=\F$ for all nonempty proper subsets $I'$ of $I$ and $S_{I,J'}(\F)=\F$ for all nonempty proper subsets $J'$ of $J$ (note that we necessarily have $\max(I') < \min(J)$ and $\max(I) < \min(J')$). Let $\G$ be the $(n,m)$-family $S_{I,J}(\F)$. Then
\[|\G\ud|+|\G\ua \cap \G\da| = |\G\ua|+|\G\da| \leq |\F\ua| + |\F\da| = \Phi(n,m)+m\]
where the inequality follows by Lemma~\ref{L:shiftfam} and the latter equality follows because
$\F$ is a convex $(n,m)$-witness family. Thus, because we know  $|\G\ud| \geq \Phi(n,m)$ and $|\G\ua \cap \G\da| \geq m$, we must in fact have equality in both cases and hence that $\G$ is a convex $(n,m)$-witness family. Furthermore, since $\G=S_{I,J}(\F) \neq \F$, strictly fewer sets in $\G$ than in $\F$ contain $\max(J)$ and the same number of sets in $\G$ and in $\F$ contain $x$ for each $x \in \{\max(J)+1,\ldots,n\}$. It follows that if we iterate this procedure we will never return to the same family and hence must eventually obtain an  $(n,m)$-witness family that is convex and strongly shifted.
\end{proof}

\section{Proof of Theorem~\ref{T:isoperRecursive} and Corollary~\ref{C:isoperBounds}}\label{S:mainProof}

With the aid of Lemma~\ref{L:canMinFamSF}, we can finish proving Theorem~\ref{T:isoperRecursive}. Using the definition from Theorem~\ref{T:totalOrdering}, we have $\C_{n,n-2}=\{A \subseteq [n]: \{1\} \subseteq A \subseteq [n-1]\}$. This family will play a key role in this section and we abbreviate it as $\C$ throughout (the value of $n$ will always be clear from context). Firstly we give a lower bound on the size of $\F\ud$ for certain families $\F$ that are not subfamilies of $\C$.

\begin{lemma}\label{L:noSetsWith1}
Let $\F$ be a strongly shifted $(n,m)$-family such that $m \leq 2^{n-2}$. If $\F \nsubseteq \C$, then $|\F\ud| \geq 2^{n-1}+m$.
\end{lemma}

\begin{proof}
Suppose that $\F \nsubseteq \C$. We may assume that neither $\emptyset$ nor $[n]$ is in $\F$ for otherwise $\F\ud = 2^{[n]}$ and the result obviously holds. Since $\F \nsubseteq \C$ we must be in one of the following three cases. \smallskip

\noindent {\textbf{Case 1.}} Suppose that $1 \notin F'$ and $n\in F''$ for some $F',F'' \in \F$ (possibly $F'=F''$). Because $S_{I,J}(\F)=\F$ where $I = \{1\}$ and $J = F'$, we must have that $(F'\setminus J) \cup I = \{1\}$ is in $\F$ (note that $F' \neq \emptyset$). Similarly, because $S_{I,J}(\F)=\F$ where $I = [n-1] \setminus F''$ and $J = \{n\}$, we must have that $(F'' \setminus J) \cup I = [n-1]$ is in $\F$ (note that $I \neq \emptyset$ since $F'' \neq [n]$). Now $\{\{1\}\}\ua \cup \{[n-1]\}\da = \{A \subseteq [n]:1 \in A \text{ or } n \notin A\}$ and hence $|\F\ud| \geq 3\cdot2^{n-2}$. So the result follows since $m \leq 2^{n-2}$.\smallskip

\noindent {\textbf{Case 2.}} Suppose that $1 \notin F'$ for some $F' \in \F$ and that no set in $\F$ contains $n$. As in Case~1, for $I = \{1\}$ and $J = F'$, we must have that $(F'\setminus J) \cup I = \{1\}$ is in $\F$ (note that $F' \neq \emptyset$). Then $\{\{1\}\}\ua \subseteq \F\ud$ and $|\{\{1\}\}\ua|=2^{n-1}$. We will complete the proof by showing that $|\F\ud \setminus \{\{1\}\}\ua| \geq m$.  Let $\varphi: \F \rightarrow \F\ud \setminus \{\{1\}\}\ua$ be the function defined by
\[\varphi(A)=
\left\{
  \begin{array}{ll}
    A \cup \{n\} \quad& \hbox{if $1 \notin A$} \\
    A \setminus \{1\}  & \hbox{if $1 \in A$.}
  \end{array}
\right.\]
For each $A \in \F$, $\varphi(A)$ is indeed in $\F\ud \setminus \{\{1\}\}\ua$ because $1 \notin \varphi(A)$,  $A \subseteq \varphi(A)$ if $1 \notin A$ and $\varphi(A) \subseteq A$ if $1 \in A$. Furthermore,  $\varphi$ is an injection because no set in $\F$ contains $n$. Thus $|\F\ud \setminus \{\{1\}\}\ua| \geq m$ and the proof is complete.\smallskip

\noindent {\textbf{Case 3.}} Suppose that each set in $\F$ contains $1$ and that $n \in F''$ for some $F'' \in \F$. The argument follows that of Case~2 closely. For $I = [n-1] \setminus F''$ and $J = \{n\}$, we must have that $(F'\setminus J) \cup I = [n-1]$ is in $\F$ (note that $I \neq \emptyset$ since $F'' \neq [n]$). Then $\{[n-1]\}\da \subseteq \F\ud$ and $|\{[n-1]\}\da|=2^{n-1}$. We will show that $|\F\ud \setminus \{[n-1]\}\da| \geq m$.  Let $\varphi: \F \rightarrow \F\ud \setminus \{[n-1]\}\da$ be the function defined by
\[\varphi(A)=
\left\{
  \begin{array}{ll}
    A \setminus \{1\} \quad& \hbox{if $n \in A$} \\
    A \cup \{n\}  & \hbox{if $n \notin A$.}
  \end{array}
\right.\]
For each $A \in \F$, $\varphi(A)$ is indeed in $\F\ud \setminus \{[n-1]\}\da$ because $n \in \varphi(A)$, $\varphi(A) \subseteq A$ if $n \in A$ and $A \subseteq \varphi(A)$ if $n \notin A$. Furthermore,  $\varphi$ is an injection because each set in $\F$ contains $1$. So the proof is complete.
\end{proof}

Next, we determine $|\F\ud|$ for families $\F$ of a particular form. We only require a very special case of this result for our proof of Theorem~\ref{T:isoperRecursive}, but the more general version we prove here will be useful later.

\begin{lemma}\label{L:joinResults}
Let $n$ and $k$ be integers with $0 \leq k \leq n$. Let $\F_1$ be a convex family of subsets of $[k]$ such that $[k] \in \F_1$, let $\F_2$ be a convex family of subsets of $[n] \setminus [k]$ such that $\emptyset \in \F_2$, and let $\F=\{F_1 \cup F_2: F_1 \in \F_1, F_2 \in \F_2\}$ be a family of subsets of $[n]$. Then $|\F|=|\F_1||\F_2|$ and
\[|\F\ud| = 2^{n-k}|\F_1|+2^{k}|\F_2|-|\F_1||\F_2|.\]
\end{lemma}

\begin{proof}
Obviously $|\F|=|\F_1||\F_2|$. Because $[k] \in \F_1$ and $\F_1$ is convex, $U \cap [k] \in \F_1$ for each $U \in \F\ua$. Thus, since $\emptyset \in \F_2$, we have that $\F\ua=\{A \cup B:A \in \F_1,B \subseteq [n] \setminus [k]\}$ and hence that $|\F\ua|=2^{n-k}|\F_1|$. Similarly, because $\emptyset \in \F_2$ and $\F_2$ is convex, $D \setminus [k] \in \F_2$ for each $D \in \F\da$. Thus, since $[k] \in \F_1$, we have that $\F\da=\{A \cup B:A \subseteq [k],B \in \F_2\}$ and hence that $|\F\da|=2^k|\F_2|$. Furthermore, we can see that $\F\ua \cap \F\da = \F$. So $|\F\ud|=|\F\ua|+|\F\da|-|\F|$ and the result follows.
\end{proof}

We can now show that, for each $m \leq 2^{n-2}$, there exists a convex $(n,m)$-witness family that has the form given by Lemma~\ref{L:sandwich}. With this, we can finally prove Theorem~\ref{T:isoperRecursive}.

\begin{lemma}\label{L:niceMinimalFamily}
Let $m$ and $n$ be nonnegative integers such that $m \leq 2^{n-2}$. There is an $(n,m)$-witness family such that $\F$ is convex and $\F \subseteq \C$.
\end{lemma}

\begin{proof}
The result is trivial for $m=0$, so assume $m \geq 1$. Note that $|\C|=2^{n-2} \geq m$. So by beginning with $\C$ and iteratively deleting minimal sets, we can obtain a convex $(n,m)$-family $\G$ such that $\G \subseteq \C$ and $[n-1] \in \G$. Thus $|\G\ud|=2^{n-1}+m$ by Lemma~\ref{L:joinResults} with $k=n-1$, $\F_1=\G$ and $\F_2=\{\emptyset\}$. So $\Phi(n,m) \leq 2^{n-1}+m$. If $\Phi(n,m) = 2^{n-1}+m$, then the proof is complete by the existence of $\G$, so we may suppose that $\Phi(n,m) < 2^{n-1}+m$. By Lemma~\ref{L:canMinFamSF} there exists a strongly shifted convex $(n,m)$-witness family $\F$. Because $|\F\ud| < 2^{n-1}+m$, we have $\F \subseteq \C$ by Lemma~\ref{L:noSetsWith1}.
\end{proof}

\begin{proof}[\textbf{\textup{Proof of Theorem~\ref{T:isoperRecursive}.}}]
We may suppose that $n \geq 2$ since it is easy to check the values of $\Phi(n,m)$ given for $n \in \{0,1\}$. We consider two cases according to the value of $m$.\smallskip

\noindent {\textbf{Case 1.}} Suppose that $m \in \{0,\ldots,2^{n-2}\}$. By Lemma~\ref{L:sandwich} it suffices to show that $\Phi(n,m) \geq 2\Phi(n-2,m)+m$. By Lemma~\ref{L:niceMinimalFamily} there is an $(n,m)$-witness family $\F$ such that $\F$ is convex and $\F\ \subseteq \C$. Let $\P = \{F \setminus \{1\}:F \in \F\}$ be a family of subsets of $\{2,\ldots,n-1\}$. Because $\F\ \subseteq \C$ and $\F$ is convex, we have that $|\P|=m$ and $\P$ is convex. So
\[\Phi(n,m)=|\F\ud| = 2|\P\ud|+m \geq 2\Phi(n-2,m)+m\]
where the first equality holds because $\F$ is an $(n,m)$-witness family, the second equality holds by Lemma~\ref{L:sandwich}, and the inequality holds by the definition of $\Phi(n-2,m)$.\smallskip

\noindent {\textbf{Case 2.}} Suppose that $m \in \{2^{n-2}+1,\ldots,2^n\}$. Let $s$ be as defined in the theorem. Lemma~\ref{L:noSetsWith1} implies $\Phi(n,2^{n-2}) \geq 3\cdot2^{n-2}$. Thus $s < 2^{n-2}$ since $m \geq 2^{n-2}+1$. So the result follows by Lemma~\ref{L:selfConjugate}.
\end{proof}

Proving Corollary~\ref{C:isoperBounds} from Theorem~\ref{T:isoperRecursive} requires some care but no new ideas.

\begin{proof}[\textbf{\textup{Proof of Corollary~\ref{C:isoperBounds}}}]
First observe that $m \leq 2^n$ implies $m \leq \sqrt{2^n m}$ and hence we have
\[\mfrac{\sqrt{2^n}}{f(n,m)}=\mfrac{\sqrt{2^n}}{\sqrt{2^{n+2}m}-m} \leq \mfrac{\sqrt{2^n}}{\sqrt{2^{n+2}m}-\sqrt{2^{n}m}} = \mfrac{1}{\sqrt{m}}.\]
So we have $f(n,m)+\sqrt{2^n} \sim f(n,m)$ as $n \rightarrow \infty$ with $m=\omega(1)$.

It only remains to show that $f(n,m) \leq \Phi(n,m) \leq  f(n,m)+\sqrt{2^{n}}$. We do this by induction on $n$. It is routine to check that the result holds for $n \in \{0,1\}$, so assume that $n \geq 2$. For brevity, we say that $m$ is \emph{nice} if $n$ is even and $m$ is a perfect square or if $n$ is odd and $m$ is twice a perfect square. Note that $m$ is nice if and only if $f(n,m)=\sqrt{2^{n+2}m}-m$ is an integer.\smallskip

\noindent {\textbf{Case 1.}} Suppose that $m \in \{0,\ldots,2^{n-2}\}$. Then, by induction,
\[f(n-2,m) \leq \Phi(n-2,m) \leq f(n-2,m)+\sqrt{2^{n-2}}.\]
and, since $f(n,m)=2f(n-2,m)+m$, the result follows easily by applying Theorem~\ref{T:isoperRecursive}. Furthermore, if $m$ is nice, then $\Phi(n-2,m)=f(n-2,m)$ by induction and this implies $\Phi(n,m)=f(n,m)$.\smallskip

\noindent {\textbf{Case 2.}} Suppose that $m \in \{2^{n-2}+1,\ldots,2^n\}$. Let $s$ be as defined in Theorem~\ref{T:isoperRecursive}. Throughout this case we will make use of the easily checked fact that $f(n,x)$ is increasing in $x$ on the interval $[0,2^{n}]$. By Theorem~\ref{T:isoperRecursive}, to establish $\Phi(n,m) \geq f(n,m)$ it suffices to show that $s \leq s_0$ where
\begin{equation}\label{E:s0def}
  s_0 = 2^n-f(n,m) = 2^n-\sqrt{2^{n+2}m}+m = \bigl(\sqrt{2^n}-\sqrt{m}\,\bigr)^2.
\end{equation}
Note that $s_0<2^{n-2}$ because $m > 2^{n-2}$ and $f(n,2^{n-2})=3\cdot2^{n-2}$. By the definition of $s$, to show that $s \leq s_0$ it suffices to show that $\Phi(n,\lfloor s_0+1 \rfloor) > 2^n-m$. This is indeed the case, since
\begin{equation}\label{E:littlePhiBound}
\Phi\bigl(n,\lfloor s_0+1 \rfloor\bigr) \geq f(n,\lfloor s_0+1 \rfloor\bigr) > f(n,s_0) = \sqrt{2^{n+2}s_0}-s_0 = 2^{n}-m
\end{equation}
where the first inequality follows from Case~1 and the final equality follows by substituting $\sqrt{s_0}=\sqrt{2^n}-\sqrt{m}$ and $s_0 = 2^n-\sqrt{2^{n+2}m}+m$  by \eqref{E:s0def} and then simplifying. Furthermore, if $m$ is nice then  \eqref{E:s0def} implies that $s_0$ is nice and so from Case~1 we have $\Phi(n,s_0)=f(n,s_0)=2^n-m$ and $\Phi(n,s_0+1) \geq f(n,s_0+1) > 2^n-m$. This shows that $s=s_0$ and hence that $\Phi(n,m)=2^n-s_0=f(n,m)$.

By Theorem~\ref{T:isoperRecursive}, to establish $\Phi(n,m) \leq f(n,m)+\sqrt{2^n}$ it suffices to show that $s \geq s_0-\sqrt{2^n}$. This holds trivially if $s_0 \leq \sqrt{2^n}$, so we may suppose otherwise. So, by the definition of $s$, to show that $s \geq s_0-\sqrt{2^n}$ it suffices to show that $\Phi(n,\lceil s_0-\sqrt{2^n}\, \rceil) \leq 2^n-m$. Now
\begin{multline}\label{E:littlePhiBound2}
  \Phi\bigl(n,\bigl\lceil s_0-\sqrt{2^n}\,\bigr\rceil) \leq f(n,\bigl\lceil s_0-\sqrt{2^n}\,\bigr\rceil)+\sqrt{2^n} <\\ f(n,s_0-\sqrt{2^n}+1)+\sqrt{2^n}= \sqrt{2^{n+2}\bigl(s_0-\sqrt{2^n}+1\bigr)}-s_0-1+\sqrt{2^{n+2}}
\end{multline}
where the first inequality follows from Case~1. Now we have $(s_0-\sqrt{2^n}+1)^{1/2}<\sqrt{s_0}-1$ because squaring both sides and simplifying reveals this to be equivalent to $s_0<2^{n-2}$, which we know holds. Using this fact in \eqref{E:littlePhiBound2}, simplifying and recalling the final equality in \eqref{E:littlePhiBound} we have, as required,
\[\Phi\bigl(n,\bigl\lceil s_0-\sqrt{2^n}\,\bigr\rceil) \leq  \sqrt{2^{n+2}s_0}-s_0-1 = 2^{n}-m-1.\qedhere\]
\end{proof}

The term $\sqrt{2^n}$ in the upper bound in Corollary~\ref{C:isoperBounds} is chosen for convenience; we have not attempted to optimise it. However, we note that it is easy to deduce from Theorem~\ref{T:isoperRecursive} that $\Phi(n,m)=f(n,m)+(\frac{3}{2}\sqrt{2}-2)\sqrt{2^{n}}$ when $m=1$ and $n$ is odd. Thus the term cannot in general be decreased by more than a constant fraction. We now state an obvious consequence of Theorem~\ref{T:isoperRecursive} that will be useful in what follows.

\begin{lemma}\label{L:strictlyMonotonic}
For any integer $n \geq 2$ we have
\[0=\Phi(n,0) < \Phi(n,1) < \cdots < \Phi(n,2^{n-2})=3\cdot2^{n-2}.\]
\end{lemma}

\begin{proof}
Since $\Phi(n-2,0) \leq \Phi(n-2,1) \leq \cdots \leq \Phi(n-2,2^{n-2})$, it follows directly from Theorem~\ref{T:isoperRecursive} that $\Phi(n,0) < \Phi(n,1) < \cdots < \Phi(n,2^{n-2})$. By the equality asserted by Corollary~\ref{C:isoperBounds}, we have $\Phi(n,2^{n-2})=3\cdot2^{n-2}$.
\end{proof}

We conclude this section with Corollary~\ref{C:isoperExplicit} below, in which we give another upper bound on $\Phi(n,m)$ that is tight in a range of cases.

\begin{corollary}\label{C:isoperExplicit}
Let $n$ and $m$ be nonnegative integers such that $m \leq 2^n$. For each $m \in \{0,\ldots,2^n\}$, we have
\[\Phi(n,m) \leq \sqrt{2^{n+a}}+\sqrt{2^{n-a}}\,m-m\]
where $a$ is the least element of $\{0,\ldots,n\}$ such that $m < 2^{a+1}$ and $a \equiv n \mod{2}$. Furthermore, we have equality if $2^a-2^{\lceil a/2\rceil}-2^{\lfloor a/2 \rfloor}+2 \leq m \leq 2^a +2^{\lceil a/2 \rceil}+2^{\lfloor a/2 \rfloor}-1$.
\end{corollary}

\begin{proof}
We first establish the upper bound on $\Phi(n,m)$. Let $k=\frac{1}{2}(n+a)$. Note that $2^{k} \geq 2^{a+1} > m$ if $a<n$ and $2^{k} = 2^{n} \geq m$ if $a=n$. Thus, by beginning with $2^{[k]}$ and iteratively deleting minimal sets, we can obtain a convex family of subsets of $[k]$ such that $[k] \in \F_1$ and $|\F_1|=m$. Then applying Lemma~\ref{L:joinResults} with $\F_2=\{\emptyset\}$ gives an $(n,m)$-family $\F$ with $|\F\ud| = 2^{k}+2^{n-k}m-m$ and hence the upper bound holds.

To prove the claim of equality, we proceed by induction on $n$. It is routine to check that it holds when $n \in \{0,1\}$. So suppose that $n \geq 2$ and $2^a-2^{\lceil a/2\rceil}-2^{\lfloor a/2 \rfloor}+2 \leq m \leq 2^a +2^{\lceil a/2 \rceil}+2^{\lfloor a/2 \rfloor}-1$.\smallskip

\noindent {\textbf{Case 1.}} Suppose that $m \leq 2^{n-2}$. Then $0 \leq a \leq n-2$. So by induction we have $\Phi(n-2,m) =\sqrt{2^{n+a-2}}+\sqrt{2^{n-a-2}}m-m$ and the result follows by applying Theorem~\ref{T:isoperRecursive}.\smallskip

\noindent {\textbf{Case 2.}} Suppose that $m > 2^{n-2}$. Then $a \in \{n-2,n\}$. Let $s$ be as defined in Theorem~\ref{T:isoperRecursive}.

If $a=n-2$, then $2^{n-2} < m \leq 2^{n-2} +2^{\lceil n/2 \rceil-1}+2^{\lfloor n/2 \rfloor-1}-1$, bearing in mind the condition of this case. Let $m'=2^{n-1}-m+1$ and note that $2^{n-2}-2^{\lceil n/2 \rceil-1}-2^{\lfloor n/2 \rfloor-1}+2 \leq m' \leq 2^{n-2}$. It can be checked that this implies that $2^{n-3} \leq m'$ and hence that $n-2$ is the least element $a'$ of $\{0,\ldots,n\}$ such that $m' < 2^{a'+1}$ and $a' \equiv n \mod{2}$. So, applying what we established in Case~1 to $m'$, we have $\Phi(n,m')=2^{n-1}+m'=2^n-m+1$. Further, by Lemma~\ref{L:strictlyMonotonic}, we have $\Phi(n,m'-1) \leq 2^n-m$. It follows that $s=m'-1=2^{n-1}-m$ and hence that $\Phi(n,m)=2^{n-1}+m$ by Theorem~\ref{T:isoperRecursive} as required.

If $a = n$, then $2^{n}-2^{\lceil n/2 \rceil}-2^{\lfloor n/2 \rfloor}+2 \leq m \leq 2^n$. So we have $\Phi(n,0)=0$ obviously and $\Phi(n,1)=2^{\lceil n/2 \rceil}+2^{\lfloor n/2 \rfloor}-1>2^n-m$ from Case~1. Thus $s=0$ and $\Phi(n,m)=2^n$ by Theorem~\ref{T:isoperRecursive} as required.
\end{proof}

\section{Proof of Theorem~\ref{T:isoperQuick}}\label{S:quickProof}

Theorem~\ref{T:isoperQuick} will follow from Theorem~\ref{T:isoperRecursive}, but we first need to establish some properties of the functions $\delta_k$ defined in the statement of Theorem~\ref{T:isoperQuick}. We begin with some basic properties in Lemma~\ref{L:deltaProps} and then move on to a more involved property in Lemma~\ref{L:deltaRecursive}.

\begin{lemma}\label{L:deltaProps}\phantom{a}
\begin{itemize}
    \item[\textup{(i)}]
$\delta_k(0)=0$ and $\delta_k(k)=1$ for all integers $k \geq 1$.
    \item[\textup{(ii)}]
$2^t\delta_{2^t}(x)=x$ for all integers $t \geq 0$ and $x \in \{0,\ldots,2^t\}$.
    \item[\textup{(iii)}]
$2^t\delta_{2^t-1}(x)=x+1$ for all integers $t \geq 1$ and $x \in \{1,\ldots,2^t-1\}$.
    \item[\textup{(iv)}]
$2^t\delta_{k}(x+1)-(x+1) \geq 2^t\delta_{k}(x)-x$ for all integers $k$, $t$ and $x$ such that $1\leq k \leq 2^t$ and $x \in \{0,\ldots,k-1\}$.
\end{itemize}
\end{lemma}

\begin{proof}
An easy proof by induction on $k$ establishes (i). Induction on $t$ establishes (ii), and from this (iii) can then be proved by induction on $t$.

Observe that the inequality in (iv) is equivalent to $\delta_{k}(x+1)-\delta_k(x) \geq \frac{1}{2^t}$. This holds when $k=1$ (noting $t \geq 0$). When $k \geq 2$, applying \eqref{E:deltaDef} and simplifying, the statement to be proved is equivalent to $\delta_{\lfloor k/2 \rfloor}(x+1) - \delta_{\lfloor k/2 \rfloor}(x) \geq \frac{1}{2^{t-1}}$ if $x \leq \lfloor\frac{k}{2}\rfloor-2$, to $\delta_{\lceil k/2 \rceil}(x+1-\lfloor \frac{k}{2} \rfloor) - \delta_{\lceil k/2 \rceil}(x-\lfloor \frac{k}{2} \rfloor) \geq \frac{1}{2^{t-1}}$ if $x \geq \lfloor\frac{k}{2}\rfloor$, and to $1 - \delta_{\lfloor k/2 \rfloor}(\lfloor \tfrac{k}{2} \rfloor-1) \geq \frac{1}{2^{t-1}}$ if $x = \lfloor\frac{k}{2}\rfloor-1$ (using (i)). In each case this follows by induction, noting that $\lfloor \tfrac{k}{2} \rfloor \leq \lceil \tfrac{k}{2} \rceil \leq 2^{t-1}$ and, in the last case, that $1=\delta_{\lfloor k/2 \rfloor}(\lfloor \tfrac{k}{2} \rfloor)$.
\end{proof}

\begin{lemma}\label{L:deltaRecursive}
Let $k$, $q$ and $t$ be positive integers such that $k+q=2^t$, and let $\ell \in \{0,\ldots,k\}$. The greatest element $r$ of $\{0,\ldots,q\}$ such that $2^t\delta_{q}(r)-r \leq k-\ell$ is $q - 2^t\delta_{k}(\ell)+\ell$.
\end{lemma}

\begin{proof}
We proceed by induction on $t$. If $\ell=0$, then $q - 2^t\delta_{k}(\ell)+\ell=q$ by Lemma~\ref{L:deltaProps}(i) and the result can be seen to hold using Lemma~\ref{L:deltaProps}(i) and $2^t-q=k$. If $k=\ell=1$, then $q - 2^t\delta_{k}(\ell)+\ell=q-2^t+1=0$ by Lemma~\ref{L:deltaProps}(i) and the result can be seen to hold using Lemma~\ref{L:deltaProps}(iii). So we may assume that $\ell \geq 1$ and $k \geq 2$. If $q=1$, then $q - 2^t\delta_{k}(\ell)+\ell=1-(\ell+1)+\ell=0$ by Lemma~\ref{L:deltaProps}(iii) and the result holds by Lemma~\ref{L:deltaProps}(i), noting that $k=2^{t}-1$ and $\ell \geq 1$. So we may further assume that $q \geq 2$ and hence that $t \geq 2$.

Let $\ell \in \{0,\ldots,k\}$, and let $r$ be the greatest element of $\{0,\ldots,q\}$ such that $2^t\delta_{q}(r)-r \leq k-\ell$. We must show that $r=q - 2^t\delta_{k}(\ell)+\ell$. We will frequently and tacitly use the fact that $k+q=2^t$. We first note that
\begin{equation}\label{E:midway}
2^t\delta_{q}(\lfloor\tfrac{q}{2}\rfloor)-\lfloor\tfrac{q}{2}\rfloor = 2^{t-1}-\lfloor\tfrac{q}{2}\rfloor = \lceil\tfrac{k}{2}\rceil
\end{equation}
where the first equality follows by \eqref{E:deltaDef} and Lemma~\ref{L:deltaProps}(i). We consider two cases according to the value of $\ell$.\smallskip

\noindent \textbf{Case 1.} Suppose that $\ell \in \{1,\ldots,\lfloor\frac{k}{2}\rfloor\}$. Let $k'=\lfloor\tfrac{k}{2}\rfloor$ and $q'=\lceil\tfrac{q}{2}\rceil$, and note that $k'+q'=2^{t-1}$. Applying \eqref{E:deltaDef}, along with Lemma~\ref{L:deltaProps}(i) if $\ell = \lfloor\frac{k}{2}\rfloor$, we must show that $r=q - 2^{t-1}\delta_{k'}(\ell)+\ell$. Since $k-\ell \geq \lceil\frac{k}{2}\rceil$ we must have $r = \lfloor\tfrac{q}{2}\rfloor+r'$ for some $r' \in \{0,\ldots,q'\}$ by \eqref{E:midway}. Then, using \eqref{E:deltaDef},
\[2^t\delta_{q}(r)-r=2^{t-1}+2^{t-1}\delta_{q'}(r')-(\lfloor\tfrac{q}{2}\rfloor+r')=\lceil\tfrac{k}{2}\rceil+2^{t-1}\delta_{q'}(r')-r'.\]
Thus, by the definitions of $r$ and $r'$, we have that $r'$ is the greatest element of $\{0,\ldots,q'\}$ such that $2^{t-1}\delta_{q'}(r')-r' \leq k'-\ell$. By induction $r'=q' - 2^{t-1}\delta_{k'}(\ell)+\ell$, and hence $r=q - 2^{t-1}\delta_{k'}(\ell)+\ell$ as required.\smallskip

\noindent \textbf{Case 2.} Suppose that $\ell \in \{\lfloor\frac{k}{2}\rfloor+1,\ldots,k\}$. Let $k'=\lceil\tfrac{k}{2}\rceil$, $q'=\lfloor\frac{q}{2}\rfloor$ and $\ell'=\ell-\lfloor\tfrac{k}{2}\rfloor$, and note that $k'+q'=2^{t-1}$. Applying \eqref{E:deltaDef}, we must show that $r$ equals
\[q - 2^{t-1} -2^{t-1}\delta_{k'}(\ell')+\lfloor\tfrac{k}{2}\rfloor+\ell'=q' - 2^{t-1}\delta_{k'}(\ell')+\ell'.\]
Since $k-\ell \leq \lceil\frac{k}{2}\rceil-1$ we must have $r \leq \lfloor\tfrac{q}{2}\rfloor-1$ by \eqref{E:midway} and Lemma~\ref{L:deltaProps}(iv). So $2^t\delta_{q}(r)-r=2^{t-1}\delta_{q'}(r)-r$ by \eqref{E:deltaDef} and since  $k'-\ell'=k-\ell$, we have that $r$ is the greatest element of $\{0,\ldots,q'\}$ such that $2^{t-1}\delta_{q'}(r)-r \leq k'-\ell'$. Thus $r=q' - 2^{t-1}\delta_{k'}(\ell')+\ell'$ by induction as required.
\end{proof}

We are now ready to prove Theorem~\ref{T:isoperQuick}.

\begin{proof}[\textbf{\textup{Proof of Theorem~\ref{T:isoperQuick}.}}]
We proceed by induction on $n$. It is routine to check that the result holds for $n \in \{0,1\}$, so fix integers $n \geq 2$ and $m \in \{1,\ldots,2^n\}$. As in the theorem statement, let $\kappa=1$ if $n$ is even and $\kappa=2$ if $n$ is odd, and let $c$ be the positive integer such that $\kappa c(c-1) \leq m < \kappa c(c+1)$. Let $\nu=\lfloor \frac{n}{2} \rfloor$ and note that $2^\nu=\sqrt{2^n/\kappa}$. For brevity, for positive integers $a$ and $x$ such that $\kappa a(a-1) \leq x \leq \kappa a(a+1)$, we let
\[\varphi_a(x)=2a-1+2\delta_{2 \kappa a}\bigl(x-\kappa a(a-1)\bigr)\]
so that the theorem claims that $\Phi(n,m)=\kappa2^\nu\varphi_c(m)-m$.

\noindent \textbf{Case 1.}  Suppose that $m \leq 2^{n-2}$. Then by induction $\Phi(n-2,m)=\kappa2^{\nu-1}\varphi_c(m)-m$. So by Theorem~\ref{T:isoperRecursive} we have, as required,
\[\Phi(n,m)=2\Phi(n-2,m)+m=2(\kappa2^{\nu-1}\varphi_c(m)-m)+m=\kappa2^\nu\varphi_c(m)-m.\]
In particular, this implies that
\begin{equation}\label{E:phi1}
\Phi(n,1)= (\kappa+1)2^{\nu}-1
\end{equation}
and this will be important to note when considering the remaining two cases.\smallskip

\noindent \textbf{Case 2.}  Suppose that $2^{n-2} < m  \leq 2^n-\Phi(n,1)$.
As in Theorem~\ref{T:isoperRecursive}, let $s$ be the greatest element of $\{0,\ldots,2^{n-2}\}$ such that $\Phi(n,s) \leq 2^n-m$. Observe that $1 \leq s < 2^{n-2}$ by the condition of this case and the fact that $\Phi(n,2^{n-2})=3\cdot2^{n-2}$ by Lemma~\ref{L:strictlyMonotonic}. Define integers $d = 2^\nu-c$, $m_0=\kappa c (c-1)$, $m_1=\kappa c (c+1)$, $s_0=\kappa d (d-1)$ and $s_1=\kappa d (d+1)$. By the conditions of this case and \eqref{E:phi1} we can deduce that $2^{\nu-1} \leq c \leq 2^{\nu}-1$ and hence that $1 \leq d \leq 2^{\nu-1}$.

Now, for any $s' \in \{0,\ldots,s_1-s_0\}$ such that $1 \leq s_0+s' \leq 2^{n-2}$, we have
\begin{equation}\label{E:sPhi}
\Phi(n,s_0+s')=\kappa2^{\nu}\varphi_d(s_0+s')-(s_0+s')=2^n-m_1+\kappa2^{\nu+1}\delta_{s_1-s_0}(s')-s'
\end{equation}
where the first equality has been proved in Case 1 (when $s'=s_1-s_0$ note that $\varphi_{d}(s_1)=\varphi_{d+1}(s_1)=2d+1$ by Lemma~\ref{L:deltaProps}(i)) and the second equality can be verified by applying the definition of $\varphi_d$ and substituting for $s_0$ and $d$. When $d=1$ we have $s \geq 1 \geq s_0$ and, otherwise, setting $s'=0$ in \eqref{E:sPhi} gives $\Phi(n,s_0) = 2^n-m_1$ by Lemma~\ref{L:deltaProps}(i) and hence $s \geq s_0$ because $m < m_1$. When $d=2^{\nu-1}$ we have $s < 2^{n-2} < s_1$. Otherwise, setting $s'=s_1-s_0$ in \eqref{E:sPhi} shows that
\[\Phi(n,s_1) = 2^n-m_1+\kappa2^{\nu+1}-s_1+s_0=2^n-m_0\]
where the first equality follows by Lemma~\ref{L:deltaProps}(i) and the second follows because $m_1-m_0+s_1-s_0=2\kappa(c+d)=\kappa2^{\nu+1}$. Thus, using  Lemma~\ref{L:strictlyMonotonic}, we have $s \leq s_1$ because $m \geq m_0$, and we also have $s_1 < 2^{n-2}$ since $d \leq 2^{\nu-1}-1$. So in all cases $s_0 \leq s \leq s_1$ and, as we already saw, $s < 2^{n-2}$.

Let $s''=s-s_0$ and note that $0 \leq s'' \leq s_1-s_0$ and $s_0+s'' < 2^{n-2}$. Thus, from \eqref{E:sPhi}, $\Phi(n,s) \leq 2^n-m$ is equivalent to $\kappa2^{\nu+1}\delta_{s_1-s_0}(s'')-s'' \leq m_1-m$. Further, if $s'' < s_1-s_0$, then $\Phi(n,s+1) > 2^n-m$ is equivalent to $\kappa2^{\nu+1}\delta_{s_1-s_0}(s''+1)-(s''+1) > m_1-m$. So, bearing in mind Lemma~\ref{L:deltaProps}(iv), it follows from the definition of $s$ that $s''$ is the greatest element $s'$ of $\{0,\ldots,s_1-s_0\}$ such that $\kappa2^{\nu+1}\delta_{s_1-s_0}(s')-s' \leq m_1-m$. Applying Lemma~\ref{L:deltaRecursive} with $t=\lceil\frac{n}{2}\rceil+1$, $k=m_1-m_0$, $q=s_1-s_0$ and $\ell=m-m_0$ we have that $s'' = s_1-s_0-\kappa2^{\nu+1}\delta_{m_1-m_0}(m-m_0)+m-m_0$ (note $2^t=\kappa2^{\nu+1}$). Hence $s=s_1-\kappa2^{\nu+1}\delta_{m_1-m_0}(m-m_0)+m-m_0$. By Theorem~\ref{T:isoperRecursive}, $\Phi(n,m)=2^n-s$ and, by substituting for $s$, $s_1$, $m_0$, $m_1$ and $d$, this is equivalent to $\Phi(n,m)=\kappa2^{\nu}\varphi_c(m)-m$ as asserted by the theorem.\smallskip

\noindent \textbf{Case 3.}  Suppose that $m > 2^n-\Phi(n,1)$. Then $s=0$ by definition and hence $\Phi(n,m)=2^n$ by Theorem~\ref{T:isoperRecursive}. So we only need to show that $\kappa2^\nu\varphi_c(m)-m=2^n$.  Now,
\begin{equation}\label{E:case3Calc}
\kappa2^{\nu}\varphi_c(m)=
\left\{
  \begin{array}{ll}
    2^{n+1}-\kappa2^\nu+\kappa2^{\nu+1}\delta_{\kappa2^{\nu+1}}(m-2^n+\kappa2^{\nu}) & \hbox{if $c=2^{\nu}$} \\
    2^{n+1}-3\kappa2^\nu+\kappa2^{\nu+1}\delta_{\kappa2^{\nu+1}-2\kappa}(m-2^n+3\kappa2^{\nu}-2\kappa) & \hbox{if $c=2^{\nu}-1$.}
  \end{array}
\right.
\end{equation}
When $c=2^{\nu}$, we have $\kappa2^{\nu+1}\delta_{\kappa2^{\nu+1}}(x)=x$ for each $x \in \{0,\ldots,\kappa2^{\nu+1}\}$ by Lemma~\ref{L:deltaProps}(ii), and using this fact in \eqref{E:case3Calc} shows that $\kappa2^\nu\varphi_c(m)-m=2^n$ as required. So we can suppose that $c<2^{\nu}$. Then, using the condition of this case and \eqref{E:phi1}, we have that $c = 2^\nu-1$ and that $m-2^n+3\kappa2^{\nu}-2\kappa$ is at least $2^\nu$ if $n$ is even and is at least $3\cdot2^\nu-2$ if $n$ is odd. Using \eqref{E:deltaDef} and Lemma~\ref{L:deltaProps}(iii) we can deduce $2^{\nu+1}\delta_{2^{\nu+1}-2}(x)=x+2$ for each $x \in \{2^{\nu},\ldots,2^{\nu+1}-2\}$ and $2^{\nu+2}\delta_{2^{\nu+2}-4}(x)=x+4$ for each $x \in \{3\cdot2^\nu-2,\ldots,2^{\nu+2}-4\}$. Considering cases according to the parity of $n$ and using these facts in \eqref{E:case3Calc}, we again have $\kappa2^\nu\varphi_c(m)-m=2^n$ as required.
\end{proof}

\section{Proof of Theorem~\ref{T:totalOrdering}}\label{S:orderProof}

Let $n \geq 2$ be an integer. To prove Theorem~\ref{T:totalOrdering}, we will recursively construct the chain $\F_0 \subsetneq \cdots \subsetneq \F_{2^n}$ of families of subsets of $[n]$ with the appropriate properties from a chain $\P_0 \subsetneq \cdots \subsetneq \P_{2^{n-2}}$ of families of subsets of $[n-2]$ with the appropriate properties. For $m \in \{0,\ldots,2^{n-2}\}$ we will form $\F_m$ by applying Lemma~\ref{L:sandwich} to $\P_m$ (with the ground set suitably relabelled). Then the key families amongst those in $\{\F_{2^{n-2}+1},\ldots,\F_{2^{n}}\}$ will be formed by taking the conjugates of the families in $\{\F_0,\ldots,\F_{2^{n-2}}\}$ according to a certain notion of conjugacy which we now proceed to define.

For any family $\F$ of subsets of $[n]$, we denote by $\F^*$ the family $2^{[n]} \setminus (\rho(\F))\ud$. We think of $\F^*$ as the conjugate of $\F$. Notice that this notation is consistent with our definition of $\C^*_{n,a}$ from the introduction because
\begin{align}
  \C_{n,a} &= \{A \subseteq [n]: [\tfrac{n-a}{2}] \subseteq A \subseteq [\tfrac{n+a}{2}]\} \notag \\
  (\C_{n,a})\ud &= \{A \subseteq [n]: [\tfrac{n-a}{2}] \subseteq A \text{ or } A \subseteq [\tfrac{n+a}{2}]\} \notag \\
  \rho((\C_{n,a})\ud) &= \{A \subseteq [n]: \{\tfrac{n+a}{2}+1,\ldots,n\} \subseteq A \text{ or } A \subseteq \{\tfrac{n-a}{2}+1,\ldots,n\}\} \notag \\
  2^{[n]} \setminus \rho((\C_{n,a})\ud) &= \{A \subseteq [n]: \{\tfrac{n+a}{2}+1,\ldots,n\} \nsubseteq A \text{ and } A \nsubseteq \{\tfrac{n-a}{2}+1,\ldots,n\}\}. \label{E:coreConj}
\end{align}
Further, for any integer $\ell \in \{0,\ldots,2^{n-2}\}$, we abbreviate $2^n-\Phi(n,\ell)$ to $\ell^*$, where the value of $n$ will always be clear from context. We observe some basic properties of our conjugacy transformation.

\begin{lemma}\label{L:conjProps}
Let $\F_\ell$ be an $(n,\ell)$-witness family with $\ell < 2^{n-2}$. Then $\F^*_\ell$ is a convex $(n,\ell^*)$-witness family and $\Phi(n,m)=2^n-\ell$ for each $m \in \{(\ell+1)^*+1,\ldots,\ell^*\}$.
\end{lemma}

\begin{proof}
Given the definitions of $\ell^*$ and $(\ell+1)^*$, saying that $\Phi(n,m)=2^n-\ell$ for each $m \in \{(\ell+1)^*+1,\ldots,\ell^*\}$ is simply a restatement of the second part of the recurrence in Theorem~\ref{T:isoperRecursive}. Note that the set $\{(\ell+1)^*+1,\ldots,\ell^*\}$ is nonempty by Lemma~\ref{L:strictlyMonotonic}.

Since $|(\F_\ell)\ud|=\Phi(n,\ell)$, we have $|\F^*_\ell|=2^n-\Phi(n,\ell)=\ell^*$ by the definition of $\F^*_\ell$. Since no set in $\F^*_\ell$ is in $(\rho(\F_\ell))\ud$, no set in $\rho(\F_\ell)$ is in $(\F^*_\ell)\ud$ and hence $|(\F^*_\ell)\ud| \leq 2^n-\ell = \Phi(n,\ell^*)$. So $\F^*_\ell$ is an $(n,\ell^*)$-witness family. Finally, if $\F^*_\ell$ were not convex then there would be sets $F_1, F_2 \in 2^{[n]} \setminus (\rho(\F_\ell))\ud$ and $A \in (\rho(\F_\ell))\ud$ such that $F_1 \subseteq A \subseteq F_2$. But this is impossible since $F_1 \in (\rho(\F_\ell))\da$ if $A \in (\rho(\F_\ell))\da$ and $F_2 \in (\rho(\F_\ell))\ua$ if $A \in (\rho(\F_\ell))\ua$.
\end{proof}

The families obtained via Lemma~\ref{L:sandwich} and their conjugates will form a subchain of the chain we require, but there will be spaces that need to be filled in. The final tool that we need for our proof of Theorem~\ref{T:totalOrdering} will allow us to do this.

\begin{lemma}\label{L:interpolate}
Let $\F_i$ be a convex $(n,i)$-family and $\F_j$ be a convex $(n,j)$-family such that $\F_i \subsetneq \F_j$. There exist convex families $\F_{i+1},\ldots,\F_{j-1}$ such that $|\F_m|=m$ for each $m\in \{i+1,\ldots,j-1\}$ and $\F_i \subsetneq \F_{i+1} \subsetneq \cdots \subsetneq \F_{j-1} \subsetneq \F_{j}$.
\end{lemma}

\begin{proof}
Fix a value of $i$. The result is trivially true when $j=i+1$, so we assume inductively that it holds when $j = h$ for some $h \in \{i+1,\ldots,2^n-1\}$ and show it holds when $j=h+1$. Given convex families $\F_i$ and $\F_{h+1}$, let $\F_{h}=\F_{h+1} \setminus \{A\}$ where $A$ is a set in $\F_{h+1} \setminus \F_i$ such that $A$ is minimal or maximal in $\F_{h+1}$. Such a set must exist because, since $\F_i \subsetneq \F_{h+1}$, if each minimal or maximal set in $\F_{h+1}$ was in $\F_i$, then $\F_i$ could not be convex. Note that $\F_h$ is convex because $A$ is minimal or maximal in $\F_{h+1}$. Now, by induction, there exist convex families $\F_{i+1},\ldots,\F_{h-1}$ such that $|\F_m|=m$ for each $m\in \{i+1,\ldots,h-1\}$ and $\F_i \subsetneq \F_{i+1} \subsetneq \cdots \subsetneq \F_{h-1} \subsetneq \F_{h}$. The families  $\F_{i+1},\ldots,\F_{h}$ show that the result holds for $j=h+1$.
\end{proof}

\begin{proof}[\textbf{\textup{Proof of Theorem~\ref{T:totalOrdering}.}}]
We prove the result by induction on $n$. It is easy to verify the result holds for $n=0$ and $n=1$. Assume inductively that, for some integer $n\geq 2$, there is a chain of families $\P_0 \subsetneq \cdots \subsetneq \P_{2^{n-2}}$ such that $\P_m$ is a convex $(n-2,m)$-witness family for each $m \in \{0,\ldots,2^{n-2}\}$ and $\P_{2^a}=\C_{n-2,a}$ for each $a \in \{0,\ldots,n-4\}$ such that $a \equiv n \mod{2}$ (note $|\C_{n-2,a}|=2^a$). Obviously $\P_{2^{n-2}}=\C_{n-2,n-2}=2^{[n-2]}$.

We will first define a subchain of the chain $\F_0 \subsetneq \cdots \subsetneq \F_{2^n}$ we desire. For each $\ell \in \{0,\ldots,2^{n-2}\}$, let $\F_\ell$ be the $(n,m)$-family given by $\{\{1\} \cup P':P' \in \P'_\ell\}$ where $\P'_\ell=\{\sigma(P):P \in \mathcal{P_\ell}\}$ and $\sigma(P)=\{x+1:x\in P\}$. We claim that
\begin{equation}\label{E:famSeq}
\F_0 \subsetneq \F_1 \subsetneq \cdots \subsetneq \F_{2^{n-2}} =  \F_{2^{n-2}}^* \subsetneq \F_{2^{n-2}-1}^* \subsetneq \cdots \subsetneq \F_0^*
\end{equation}
is a chain of convex witness families of subsets of $[n]$. Observe that $\F_{2^{n-2}}^*=\C_{n,n-2}^*=\C_{n,n-2}=\F_{2^{n-2}}$ using \eqref{E:coreConj}. For each $\ell \in \{0,\ldots,2^{n-2}\}$ it can be seen that $\F_\ell$ is a convex $(n,\ell)$-witness family using Theorem~\ref{T:isoperRecursive} and Lemma~\ref{L:sandwich}, and hence also that $\F^*_\ell$ is a convex $(n,\ell^*)$-witness family by Lemma \ref{L:conjProps}. So to prove our claim it remains to show \eqref{E:famSeq} is indeed a chain. Since $\P_0 \subsetneq \cdots \subsetneq \P_{2^{n-2}}$, we have $\F_0 \subsetneq \cdots \subsetneq \F_{2^{n-2}}$. Thus we also have $\rho(\F_0) \subsetneq \cdots \subsetneq \rho(\F_{2^{n-2}})$ and hence $\F_{2^{n-2}}^* \subseteq \F_{2^{n-2}-1}^* \subseteq \cdots \subseteq \F_0^*$. So, in fact, we have $\F_{2^{n-2}}^* \subsetneq \F_{2^{n-2}-1}^* \subsetneq \cdots \subsetneq \F_0^*$ because $|\F^*_\ell|=\ell^*$ since $\F^*_\ell$ is an $(n,\ell^*)$-witness family and $(2^{n-2})^*<(2^{n-2}-1)^*<\cdots<0^*$ by Lemma~\ref{L:strictlyMonotonic}. Thus \eqref{E:famSeq} is indeed a chain of convex witness families as claimed. Further, for each nonnegative integer $a \in \{0,\ldots,n-2\}$ with $a \equiv n \mod{2}$, we have $\F_{2^a}=\C_{n,a}$ since $\P_{2^a}=\C_{n-2,a}$. Thus \eqref{E:famSeq} includes the families $\C_{n,a}$ and $\C^*_{n,a}$ for each $a \in \{0,\ldots,n-2\}$ such that $a \equiv n \mod{2}$.

We now claim we can extend \eqref{E:famSeq} to the desired chain $\F_0 \subsetneq \cdots \subsetneq \F_{2^n}$ by, for each $\ell \in \{0,\ldots,2^{n-2}\}$, setting $\F_{\ell^*} = \F_{\ell}^*$ and interpolating families $\F_{(\ell+1)^*+1}, \ldots, \F_{\ell^*-1}$ between $\F_{(\ell+1)^*}$ and $\F_{\ell^*}$ such that
\[\F_{\ell+1}^* = \F_{(\ell+1)^*} \subsetneq \F_{(\ell+1)^*+1} \subsetneq \cdots \subsetneq \F_{\ell^*-1} \subsetneq \F_{\ell^*} = \F_{\ell}^*\]
is a chain of families for which $|\F_m|=m$ and $\F_m$ is convex for each integer $m$ with $(\ell+1)^* < m < \ell^*$. Lemma~\ref{L:interpolate} guarantees such families exist. For each integer $m$ with $(\ell+1)^* < m \leq \ell^*$, $\F_m$ will be a witness family because $\Phi(n,m)=2^n-\ell$ by Lemma~\ref{L:conjProps} and $|(\F_m)\ud| \leq |(\F_{\ell}^*)\ud|=2^n-\ell$  since $\F_m \subseteq \F_{\ell}^*$. So we can indeed construct a chain $\F_0 \subsetneq \cdots \subsetneq \F_{2^n}$ with the appropriate properties.
\end{proof}

\section{Cross-Sperner families}\label{S:CS}

Our results also give an answer to a problem raised by Gerbner et al.\ in \cite{GerLemPalPatSze}. Let $\F$ and $\G$ be families of subsets of $[n]$. We say the pair $(\F,\G)$ is \emph{cross-Sperner} if $F \nsubseteq G$ and $G \nsubseteq F$ for all $F \in \F$ and $G \in \G$. Pairs of cross-Sperner families have been studied in \cite[\S8]{AhlZha} and \cite{GerLemPalPatSze}. In \cite[\S1]{GerLemPalPatSze}, Gerbner et al.\ raised the question of determining the maximum possible value of $|\G|$ when $|\F|=m$ and $(\F,\G)$ is cross-Sperner. We denote this maximum value by $g(n,m)$. Observe that $(\F,\G)$ is cross-Sperner if and only if $\G \subseteq 2^{[n]} \setminus \F\ud$. From this it is clear that $g(n,m)=2^n-\Phi(n,m)$ and so our work here gives a complete solution to their question. The question of Gerbner et al.\ arose out of their study of the maximum value of $|\F|+|\G|$ under the additional condition that $\F$ and $\G$ are both nonempty (without this assumption it is easy to see that this maximum value is $2^n$). They proved Theorem~\ref{T:ImproveGerbnerEtAl} below under the assumption that $n$ is large and stated that they believed that in fact it held for all $n$. We are able to confirm their belief.

\begin{theorem}\label{T:ImproveGerbnerEtAl}
Let $(\F,\G)$ be a cross-Sperner pair of families of subsets of $[n]$ such that neither $\F$ nor $\G$ is empty. Then $n\geq 2$ and
\[|\F|+|\G| \leq 2^n-2^{\lceil n/2 \rceil} - 2^{\lfloor n/2 \rfloor}+2.\]
Furthermore, for $n \neq 3$, we have equality if and only if one $\F$ or $\G$ consists of exactly one set $A$ of
size $\lceil \frac{n}{2} \rceil$ or $\lfloor \frac{n}{2} \rfloor$ and the other family consists of all subsets of $[n]$ that are neither supersets nor subsets of $A$.
\end{theorem}

\begin{proof}
It is routine to check that $n \geq 2$ and that the result holds for $n \in \{2,3\}$, so suppose $n \geq 4$.
Suppose without loss of generality that $|\F| \leq |\G|$ and let $m=|\F|$. Note that $m \leq 2^{n-2}$ for otherwise we would have  $|\mathcal{G}| \leq 2^n-\Phi(n,m) \leq 2^n-\Phi(n,2^{n-2}) = 2^{n-2}$ using Lemma~\ref{L:strictlyMonotonic}. By the definition of $g(n,m)$, we have $|\F|+|\G| \leq g(n,m)+m$. Observe that
\begin{equation}\label{E:GerbnerEtAl}
g(n,1)+1=2^n-\Phi(n,1)+1=2^n-2^{\lceil n/2 \rceil} - 2^{\lfloor n/2 \rfloor}+2
\end{equation}
where the last equality is a rephrasing of \eqref{E:phi1}. We consider two cases according to whether $m=1$.

\noindent \textbf{Case 1.} Suppose that $m=1$. Then, $|\F|+|\G| \leq g(n,1)+1$ and the bound of the theorem holds by \eqref{E:GerbnerEtAl}. If $\F=\{A\}$ for some subset $A$ of $[n]$, then $|\F\ud|=2^{k}+2^{n-k}-1$ where $k=|A|$. Using this it is easy to confirm that we have equality in this bound if and only if $\F=\{A\}$ and $|A| \in \{\lfloor \frac{n}{2}\rfloor,\lceil \frac{n}{2} \rceil\}$.\smallskip

\noindent \textbf{Case 2.} Suppose that $m \in \{2,\ldots,2^{n-2}\}$. Since $|\F|+|\G| \leq g(n,m)+m$, by \eqref{E:GerbnerEtAl} it suffices to show that $g(n,m)+m<g(n,1)+1$ (Gerbner et al.\ were able to do this only under the assumption that $n$ is large). Using \eqref{E:GerbnerEtAl} and the lower bound of Corollary~\ref{C:isoperBounds} we can deduce that $\Phi(n,2)-2>\Phi(n,1)-1$. Furthermore, from Lemma~\ref{L:strictlyMonotonic} we have $\Phi(n,m)-m \geq \Phi(n,2)-2$. So $\Phi(n,m)-m>\Phi(n,1)-1$ and hence $g(n,m)+m<g(n,1)+1$ as required.
\end{proof}

\section{Conclusion}\label{S:conc}

It is clear from the arbitrary choices afforded by the proof of Theorem~\ref{T:totalOrdering}, that the total ordering whose existence it asserts is by no means unique. It can also be shown that the families $\mathcal{C}_{n,a}$ and $\mathcal{C}^*_{n,a}$ are not the unique witness families of their respective sizes. In fact, for even values of $n$, applying Lemma~\ref{L:joinResults} with $k=\frac{n}{2}$ and any choice of $\F_1$ and $\F_2$ such that $|\F_1|=|\F_2|=c$ will produce an $(n,c^2)$-witness family $\F$ (it can be checked that $|\F\ud|=|\Phi(n,c^2)|$ using the equality in Corollary~\ref{C:isoperBounds}). Similarly, for odd values of $n$, applying Lemma~\ref{L:joinResults} with $k=\frac{n-1}{2}$ and any choice of $\F_1$ and $\F_2$ such that $|\F_1|=c$ and $|\F_2|=2c$ will produce an $(n,2c^2)$-witness family. This allows many different witness families to be produced. This contrasts with the Kruskal-Katona problem for which it is known that the only families of $\binom{x}{k}$ $k$-sets with minimum shadow size are those that comprise all $k$-subsets of an $x$-set (see \cite{Kee} for further results). Likewise, it is known that the Hamming balls are the only families of their respective sizes for which the inequality in Harper’s theorem holds with equality (see \cite{Rat} for further results). An easy generalisation of the proof of Lemma~\ref{L:noSetsWith1} does show, however, that the families $\mathcal{C}_{n,a}$ are the unique strongly shifted witness families of their respective sizes. By analogy to Harper's theorem, it would be of interest to find a concise definition for a total ordering of the subsets of $[n]$ that induces a chain of families satisfying the conditions of Theorem~\ref{T:totalOrdering}.\bigskip

\noindent\textbf{Acknowledgments.}
Thanks to Ian Roberts for helpful discussions and comments. Adam Gowty was supported by an Australian Government Research Training Program Scholarship. Daniel Horsley was supported by Australian Research Council grants FT160100048 and DP220102212.

\pagebreak

\begin{table}[H]
\begin{center}
\begin{footnotesize}
\begin{tabular}{c|ccccc}\label{t:phitable}
$m$  & $\Phi(2,m)$ & $\Phi(3,m)$ & $\Phi(4,m)$ & $\Phi(5,m)$ & $\Phi(6,m)$ \\ \hline
$0$  & $0$         & $0$         & $0$         & $0$         & $0$         \\
$1$  & $3$         & $5$         & $7$         & $11$        & $15$        \\
$2$  & $4$         & $6$         & $10$        & $14$        & $22$        \\
$3$  & $4$         & $7$         & $11$        & $17$        & $25$        \\
$4$  & $4$         & $8$         & $12$        & $20$        & $28$            \\
$5$  &             & $8$         & $13$        & $21$        &  $31$           \\
$6$  &             & $8$         & $14$        & $22$        &  $34$           \\
$7$  &             & $8$         & $15$        & $23$        &  $37$           \\
$8$  &             & $8$         & $15$        & $24$        &  $38$           \\
$9$  &             &             & $15$        & $25$        &    $39$         \\
$10$ &             &             & $16$        & $26$        &   $42$          \\
$11$ &             &             & $16$        & $27$        &   $43$          \\
$12$ &             &             & $16$        & $28$        &   $44$          \\
$13$ &             &             & $16$        & $29$        &   $45$          \\
$14$ &             &             & $16$        & $29$        &   $46$          \\
$15$ &             &             & $16$        & $29$        &    $47$         \\
$16$ &             &             & $16$        & $30$        &   $48$          \\
$17$ &             &             &             & $30$        &   $49$          \\
$18$ &             &             &             & $30$        &   $50$          \\
$19$ &             &             &             & $31$        &   $51$          \\
$20$ &             &             &             & $31$        &   $52$          \\
$21$ &             &             &             & $31$        &   $53$          \\
$22$ &             &             &             & $32$        &  $54$           \\
$23$ &             &             &             & $32$        &  $55$           \\
$24$ &             &             &             & $32$        &  $55$           \\
$25$ &             &             &             & $32$        &  $55$           \\
$26$ &             &             &             & $32$        &  $56$          \\
$27$ &             &             &             & $32$        &  $57$           \\
$28$ &             &             &             & $32$        &  $58$           \\
$29$ &             &             &             & $32$        &  $58$           \\
$30$ &             &             &             & $32$        &  $58$           \\
$31$ &             &             &             & $32$        &  $59$           \\
$32$ &             &             &             & $32$        &  $59$           \\
$33$ &             &             &             &             &   $59$          \\
$34$ &             &             &             &             &   $60$          \\
$35$ &             &             &             &             &   $60$          \\
$36$ &             &             &             &             &   $60$          \\
$37$ &             &             &             &             &   $61$          \\
$38$ &             &             &             &             &   $61$          \\
$39$ &             &             &             &             &   $61$          \\
$40$ &             &             &             &             &   $62$          \\
$41$ &             &             &             &             &   $62$          \\
$42$ &             &             &             &             &   $62$          \\
$43$ &             &             &             &             &   $63$          \\
$44$ &             &             &             &             &   $63$          \\
$45$ &             &             &             &             &   $63$          \\
$46$ &             &             &             &             &   $63$          \\
$47$ &             &             &             &             &   $63$          \\
$48$ &             &             &             &             &   $63$          \\
$49$ &             &             &             &             &   $63$          \\
$50 - 64$ &             &             &             &             & $64$            \\

\end{tabular}
\end{footnotesize}

\vspace{1em}
\caption{Values of $\Phi(n,m)$ for $n \in \{2,3,4,5,6\}$.}\label{Tab:smallVals}
\end{center}
\end{table}

\begin{figure}[H]
\begin{center}
\includegraphics[width=\textwidth]{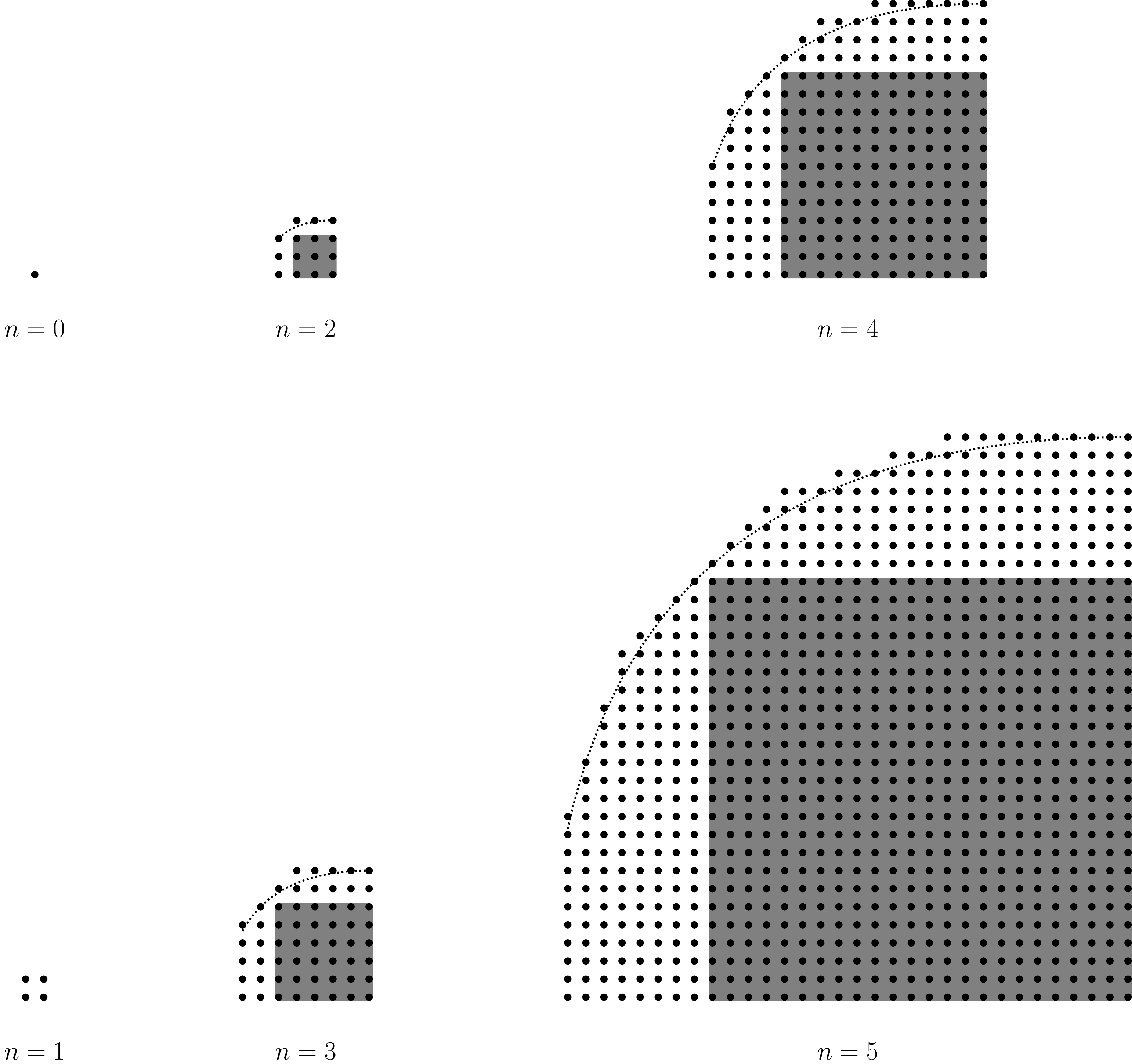}
\caption{Ferrers diagrams for the integer partition $(\Phi(n,1),\Phi(n,2),\ldots,\Phi(n,2^n))$ for $n \in \{0,1,2,3,4,5\}$. A dot is placed at the point $(x,y)$ for each $x \in \{1,\ldots,2^n\}$ and $y \in \{1,\ldots,\Phi(n,x)\}$. For $n \in \{2,3,4,5\}$ we shade the Durfee square of the partition to highlight its self-conjugacy and also plot the lower bound of Corollary~\ref{C:isoperBounds} as a dashed line.}\label{F:Ferrers}
\end{center}
\end{figure}


\begin{thebibliography}{99}

    \bibitem{AhlZha}
R. Ahlswede and Z. Zhang,
On cloud-antichains and related configurations.
\textit{Discrete Math.} \textbf{85} (1990), 225--245.

    \bibitem{AndEri}
G.E. Andrews and K. Eriksson, Integer partitions, Cambridge University Press (2004).

    \bibitem{Bas1}
M.A. Bashov, Minimization of two-sided shadow in the unit cube, \textit{Discrete Math. Appl.} \textbf{21} (2011), 517--535

    \bibitem{Bas2}
M.A. Bashov, Nonexistence of a Kruskal–Katona type theorem for double-sided shadow minimization in the Boolean cube layer, \textit{Acta Univ. Sapientiae Inform.} \textbf{5.1} (2014), 53--62.

    \bibitem{BehKupMorWri}
N. Behague, A. Kuperus, N. Morrison and A. Wright, Improved bounds for cross-Sperner systems, arXiv:2302.02516 (2023).

    \bibitem{Bol}
B. Bollob\'{a}s, Combinatorics, Cambridge University Press (1986).

    \bibitem{Fra}
P. Frankl, A new short proof for the Kruskal-Katona theorem, \textit{Discrete Math.} \textbf{48} (1984), 327--329.

    \bibitem{GerLemPalPatSze}
D. Gerbner, N. Lemons, C. Palmer, B. Patk\'{o}s and V. Sz\'{e}csi, Cross-Sperner families, \textit{Studia Sci. Math. Hungar.} \textbf{49} (2012), 44--51.

    \bibitem{Har}
L.H. Harper, Optimal numberings and isoperimetric problems on graphs, \textit{J. Combinatorial Theory} \textbf{1} (1966), 385--393.

    \bibitem{Kat}
G. Katona, A theorem of finite sets, \textit{Theory of graphs} pp. 187--207, Academic Press (1968).

    \bibitem{Kee}
P. Keevash, Shadows and intersections: stability and new proofs, \textit{Adv. Math.} \textbf{218} (2008), 1685--1703.

    \bibitem{KeeLon}
P. Keevash and E. Long, Stability for vertex isoperimetry in the cube, \textit{J. Combin. Theory Ser. B} \textbf{145} (2020), 113--144.

    \bibitem{Kru}
J.B. Kruskal, The number of simplices in a complex, in \textit{Mathematical optimization techniques} pp. 251--278, Univ. California Press (1972).

\bibitem{Lov}
L. Lov\'{a}sz, Combinatorial problems and exercises, 2nd edition, North-Holland, Amsterdam (1993).

    \bibitem{PrzRob}
M. Przykucki and A. Roberts, Vertex-isoperimetric stability in the hypercube,
\textit{J. Combin. Theory Ser. A} \textbf{172} (2020), 105186, 19 pp.

    \bibitem{Rat}
E. R\"{a}ty, Uniqueness in Harper's vertex-isoperimetric theorem.
\textit{Discrete Math.} \textbf{343} (2020), 111696, 16 pp.

\end{thebibliography}
\end{document}